\documentclass[11pt]{amsart}

 \usepackage{graphicx,color}
\usepackage{geometry}
\usepackage[all]{xy}
\usepackage{amssymb}

\newtheorem{introtheorem}{Theorem}
\newtheorem{introproposition}[introtheorem]{Proposition}
\swapnumbers
\newtheorem{theorem}{Theorem}[section]
\newtheorem{lemma}[theorem]{Lemma}
\newtheorem{proposition}[theorem]{Proposition}
\newtheorem{corollary}[theorem]{Corollary}
\theoremstyle{definition}
\newtheorem{definition}[theorem]{Definition}
\newtheorem{notations}[theorem]{Notations}
\newtheorem{notation}[theorem]{Notation}
\newtheorem{remark}[theorem]{Remark}
\newtheorem{remarks}[theorem]{Remarks}
\newtheorem{example}[theorem]{Example}
\newtheorem{examples}[theorem]{Examples}
\newtheorem{se}[theorem]{}
\newtheorem*{remark*}{Remark}
\newtheorem*{remarks*}{Remarks}
\newtheorem*{definition*}{Definition}

\usepackage{calrsfs}

\usepackage[T1]{fontenc} 
\usepackage{textcomp}
\usepackage{times}
 \usepackage[scaled=0.92]{helvet}

\renewcommand{\tilde}{\widetilde} 
\newcommand{\C}{\mathbf{C}}

\newcommand{\R}{\mathbf{R}}
\newcommand{\Z}{\mathbf{Z}}

\newcommand{\cZ}{\mathcal{Z}}

\newcommand{\cR}{\mathcal{R}}

\newcommand{\cC}{\mathcal{C}}
\newcommand{\dD}{\mathcal{D}}

\newcommand{\len}{\ell}
\newcommand{\T}{\mathbf{T}}
\newcommand{\tzeta}{{\tilde{\zeta}}}

\newcommand{\tr}{\mathrm{tr}}
\newcommand{\vol}{\mathrm{vol}}
\newcommand{\Tr}{\mathrm{tr}}
\newcommand{\im}{\mathrm{im}}
\newcommand{\Res}{\mathrm{Res}}
\newcommand{\Hom}{\mathrm{Hom}}
\renewcommand{\Re}{\mathsf{Re}}

\newcommand{\coker}{\mathrm{coker}}
\newcommand{\dil}{\mathrm{dil}}

\newcommand{\nn}{\nonumber}

\newcommand{\bea}          {\begin{eqnarray}}
\newcommand{\eea}          {\end{eqnarray}}
\newcommand{\beastar}          {\begin{eqnarray*}}
\newcommand{\eeastar}          {\end{eqnarray*}}

\begin{document}
 
\date{\today (version for JNCG)}
\title[The spectral length of a map]{The spectral length of a map \\ between Riemannian manifolds}
\author[G.\ Cornelissen]{Gunther Cornelissen}
\author[J.~W.\ de Jong]{Jan Willem de Jong}
\address{\normalfont{Mathematisch Instituut, Universiteit Utrecht, Postbus 80.010, 3508 TA Utrecht, Nederland}}
\email{\{g.cornelissen,j.w.w.dejong\}@uu.nl}
\thanks{We thank Erik van den Ban, Alain Connes, Nigel Higson, Henri Moscovici, Jorge Plazas and Ori Yudilevich for their helpful suggestions.}
\begin{abstract} 
\noindent To a closed Riemannian manifold, we associate a set of (special values of) a family of Dirichlet series, indexed by functions on the manifold. We study the meaning of equality of two such families of spectral Dirichlet series under pullback along a map. 
This allows us to give a spectral characterization of when a smooth diffeomorphism between Riemannian manifolds is an isometry, in terms of equality along pullback. We also use the invariant to define the {(spectral) length} of a map between Riemannian manifolds, where a map of length zero between manifolds is an isometry. We show that this length induces a distance between Riemannian manifolds up to isometry. 
\end{abstract}

\maketitle

\tableofcontents 

\newpage

\section*{Introduction}

It is well known that the spectrum of the Laplace-Beltrami operator $\Delta_Y$ on  a (closed, viz., compact without boundary) Riemannian manifold $Y$ does not necessarily capture its isometry type (cf.\ Milnor \cite{Milnor} and later work); actually, it does not even determine the homeomorphism type of the manifold (\cite{Ikeda}, \cite{Vigneras}).  Knowledge of the spectrum $\Lambda_Y = \{ \lambda \}$  (with multiplicities)  is equivalent to knowledge of the zeta function
$$ \zeta_Y(s) := \tr(\Delta^{-s}) = \sum_{0 \neq \lambda \in \Lambda_Y} \lambda^{-s}. $$
In this paper, we study what happens if one considers this zeta function as only one member of a family of zeta functions / Dirichlet series associated with the algebra of functions on $Y$.   Namely, for a function $a_0 \in C^{\infty}(Y)$, let $$\zeta_{Y,a_0}(s):=\tr(a_0\Delta_Y^{-s})$$ denote the zeta function associated to $a_0$ and $\Delta_Y$, where the trace is taken in $L^2(Y,d\mu_Y)$. It is a generalized Dirichlet series in $\lambda^{-s}$ for $\lambda\in \Lambda_Y$, and it can be extended to a meromorphic function on the complex plane. We will actually also need  the following higher order version of this zeta function, which arises naturally for example in noncommutative geometry. For functions $a_1 \in C^{\infty}(Y)$, we define $$\tzeta_{Y,a_1}(s):=\tr(a_1[\Delta_Y,a_1]\Delta_Y^{-s}).$$ These zeta functions are  diffeomorphism invariants by construction (cf.\ Lemma \ref{diffinv} for an exact statement).  We have the following relation between the zeta functions: $$\tzeta_{Y,a_1}(s):=\zeta_{Y,g_Y(da_1,da_1)}(s)$$
(thus, we could take the right hand side as a definition, but in that way we would obscure the possibility of generalizing our constructions to the case of noncommutative Riemannian geometries). 

In the first part of this paper, we study the meaning of equality of these families of zeta functions (rather than just the spectrum)  under pullback by a map: 

\begin{introtheorem} \label{Z0thm} Let $\varphi \, : \, X \rightarrow Y$ denote a smooth diffeomorphism between closed connected smooth Riemannian manifolds with smooth metric. The following are equivalent:
\begin{enumerate}
\item[\textup{(i)}] We have that
 \begin{enumerate}
\item[\textup{(a)}] $ \zeta_{Y,a_0} = \zeta_{X,\varphi^*(a_0)}$ for all $a_0 \in C^{\infty}(Y)$, and 
\item[\textup{(b)}] $ \tzeta_{Y,a_1} = \tzeta_{X,\varphi^*(a_1)}$ for all $a_1 \in C^{\infty}(Y)$.
\end{enumerate}
\medskip
\item[\textup{(ii)}] The map $\varphi$ is an isometry.
\end{enumerate} 
\end{introtheorem}

\noindent To shorten notation, if a map $\varphi \, : \, (X,g_X) \rightarrow (Y,g_Y)$ is fixed, we set $$ a^* := \varphi^*(a) $$
for $a \in C^{\infty}(Y)$, unless confusion can arise. 

\begin{remarks*}
\mbox{ }

- Since we are dealing with usual Dirichlet series (the spectrum is an increasing sequence of positive real numbers with finite multiplicities), the condition $\zeta_{Y,a_0}=\zeta_{X,a_0^*}$ is satisfied when   $\zeta_{Y,a_0}(k)=\zeta_{X,a_0^*}(k) $ for all sufficiently large integers $k$, and similarly for the higher order zeta functions (cf.\ \cite{Serre}, Section 2.2).

- The dependence of the zeta functions on $a_i$ is $\C$-linear. Thus, we may for example restrict to functions of unit norm (supremum over the manifold). 

- One can replace the condition ``$a \in C^{\infty}(Y)$'' in Theorem \ref{Z0thm} by ``$a \in A$'' for $A$ some dense subset of $C^{\infty}(Y)$.  Since we assume the manifolds compact, we can pick such a countable set. We do not know whether it is possible to pick a \emph{finite} set $A$, depending only on some topological characteristics of the manifold. 

- Combining the above three remarks, we see that we have actually found a \emph{countable} sequence of spectrally defined, diffeomorphism invariant real numbers that characterize the manifold up to isometry  \emph{in a fixed $C^{\infty}$-diffeomorphism class}: the values $\zeta_{X,a_0}(k)$ and $\tzeta_{X,a_1}(k)$ for $a_i$ in a countable basis for $C^{\infty}(Y)$, and $k$ running through all sufficiently large integers. Notice that this invariant presupposes knowledge of the manifold: one needs to be able to ``label'' by the smooth functions, and by integers $k$.  
\end{remarks*}

The proof of Theorem \ref{Z0thm} is rather formal (essentially, a suitable residue of the two-variable zeta function contains the metric tensor). By contrast, the next theorem, which studies what happens if we only have equality of the one-variable zeta functions, actually uses some analysis of PDE's and structure of nodal sets:

\begin{introtheorem} \label{extraprop}
Let $\varphi \, : \, X \rightarrow Y$ denote a smooth diffeomorphism between closed Riemannian manifolds with smooth metric. Then 
\begin{itemize}
\item In the previous theorem, it suffices in condition \textup{(b)} to have equality of \emph{one} coefficient of the Dirichlet series (for all $a_1$) for conditions \textup{(a)} and \textup{(b)} to be equivalent to \textup{(ii)}.  
\item If the spectrum of $X$ or $Y$ is simple, then condition \textup{(a)} alone is equivalent to \textup{(ii)} in the previous theorem. 
\end{itemize} 
\end{introtheorem}

\begin{remarks*}

\mbox{ } 

- We can use the method of proof to show for example that if $g$ and $g'$ are two smooth Riemannian structures on a closed connected manifold and with simple Laplace spectrum, then an equality of heat kernels ``on the diagonal'' $K_g(t,x,x)=K_{g'}(t,x,x)$ for sufficiently small $t>0$ implies that $K_g(t,x,y)=K_{g'}(t,x,y)$ for all $t>0$ (and hence $g=g'$), cf.\ Corollary \ref{verysmooth}.  

- In section \ref{tori} we compute these zeta functions for flat tori. In this case, condition (a) is equivalent to isospectrality and the fact that $\varphi$ has trivial jacobian.

\end{remarks*}

We can rephrase  Theorem \ref{Z0thm} in terms of the \emph{length of $\varphi$}, a concept that we now introduce. The basic idea is to measure in some way the distance between the zeta functions on the one manifold and the pull back by the map of the zeta function on the other manifold. 

\begin{definition*}  
Let $\varphi \, : \, X \rightarrow Y$ denote a smooth diffeomorphism of closed Riemannian manifolds of dimension $d$. Define 
$$ d_1(\varphi,a_0,a_1):= \mathop{\sup_{d \leqslant s\leqslant d+1}} \! \! \! \max \, \{ | \log \left| \frac{\zeta_{X,a^*_0}(s)}{\zeta_{Y,a_0}(s)} \right| |, | \log \left| \frac{\tzeta_{X,a^*_1}(s)}{\tzeta_{Y,a_1}(s)} \right| | \}. $$ The \emph{length of $\varphi$} is defined by $$ \ell(\varphi):= \mathop{\sup_{a_0 \in C^{\infty}(Y,\R_{\geq 0})-\{0\}}}_{a_1\in C^{\infty}(Y)-\R} \frac{d_1(\varphi,a_0,a_1)}{1+d_1(\varphi,a_0,a_1)}.$$
\end{definition*}

We discuss this notion in a more abstract framework of general \emph{length categories}, a concept we believe to be useful in non-abelian categories such as closed Riemannian manifolds up to isometry, cf.\ Section \ref{lengthcat}. For now, we just state the main properties of $\ell$:

\begin{introproposition} 
The function $\ell$ satisfies 
\begin{enumerate}
\item[\textup{(i)}] For all smooth diffeomorphisms of closed Riemannian manifolds ${X \stackrel{\varphi}{\rightarrow} Y}, {Y\stackrel{\psi}{\rightarrow} Z}$, we have $$\ell(\varphi \circ \varphi) \leqslant  \ell(\varphi) + \ell (\psi);$$
\item[\textup{(ii)}] $\ell(\varphi)=0$ if and only if $\varphi$ is an isometry.
\end{enumerate} 
\end{introproposition}

We compute some examples, such as the length of the natural map between circles of different radii, see Figure \ref{fig1}; and we show that the length of a linear map between isospectral tori is bounded in terms of its spectral norm. 

We then show  that the notion of length leads to a meaningful concept of \emph{distance between Riemannian manifolds $X,Y$} as infimum of the length of all possible maps between $X$ and $Y$. 

\begin{remarks*} 
\mbox{ }

- To consider not a single zeta function, but rather a family of zeta functions over the algebra of functions is natural in noncommutative geometry (\cite{ConnesLMP}), where the unit of the algebra of functions does not have to play a distinguished role (the underlying $C^*$-algebra  could even be non-unital): our zeta functions bear some resemblance to the construction of Hochschild homology, but with genuine traces instead of residual traces. 

We list some other manifestation of the philosophy behind our main theorem. In \cite{CM}, it is shown how to associate a spectral triple to a compact hyperbolic Riemann surface (by considering the action of a uniformizing Schottky group on the (fractal) boundary of the Poincar\'e disk), such that he following property holds: if a map between two Riemann surfaces induces equal families of zeta functions of the associated spectral triples, then the map is conformal or anti-conformal. This construction was adopted to the case of  finite \emph{graphs} in \cite{dJ}. Finally, for \emph{number fields}, see \cite{CM2}. 

- One may now wonder whether a similar theory persists to the case of spin manifolds with the Dirac operator replacing the Laplace operator, and whether it applies to noncommutative ``Riemannian geometries'', a.k.a.\ spectral triples (finitely summable).

\end{remarks*}

\part{SPECTRAL DIRICHLET SERIES}

\section{Notations and preliminaries} 

\begin{notations} To set up notation, suppose $(X,g_X)$ is a closed (i.e., compact without boundary) smooth manifold of dimension $>0$, with smooth Riemannian metric. Denote by $\mu_X$ the induced measure on $X$, and let $\Delta_X$ denote the Laplace-Beltrami operator acting on  $L^2(X)=L^2(X,\mu_X)$, with domain the smooth functions. Write $\Lambda_X$ for its spectrum with multiplicities. 

Suppose we have picked an orthonormal basis of smooth real eigenfunctions for the Laplacian on a Riemannian manifold $X$. We will use various notations depending on the context. We denote an eigenfunction in the chosen basis, with eigenvalue $\lambda$, by $\Psi_{X,\lambda}$ or $\Psi_\lambda$ if the manifold is fixed. We also write  $$\Psi_X \vdash \lambda$$ if $\Psi_X$ is an eigenfunction on $X$ in our chosen basis that belongs to the eigenvalue $\lambda$. 

Let $C^{\infty}(X)$ denote the set of smooth real-valued functions on $X$ (most of the time, one may also use complex valued functions --- this should be clear from the context). Define the zeta-function parametrized by $a_0  \in C^{\infty}(X)$ as
\bea
\zeta_{X,g_X;a_0}(s) = \zeta_{X,a_0}:=\Tr(a_0 \Delta_X^{-s})
\eea
where the complex exponent is taken in the sense of spectral theory (see Formula (\ref{zetaexpansion})); 
and the double zeta function parametrized by $a_1, a_2 \in C^{\infty}(Y)$ as 
\bea
\zeta_{X,g_X;a_1,a_2}(s)=\zeta_{X,a_1,a_2}(s)=\Tr(a_1 [ \Delta_X,a_2] \Delta_X^{-s}).
\eea
We will mainly be concerned with the diagonal version of this two-variable zeta function:
\bea
\tzeta_{X,g_X;a_1}(s):=\zeta_{X,g_X;a_1,a_1}(s)=\Tr(a_1 [ \Delta_X,a_1] \Delta_X^{-s}).
\eea
Finally, let $$K_{X,g}(t,x,y)=K_X(t,x,y)$$ denote the heat kernel of $X$. Sometimes we will write $K_g$ if we are on a fixed manifold. Otherwise, our notation will mostly suppress the metric $g$. 
We also make the convention to write in the usual way $\zeta_X=\zeta_{X,1_X}.$
\end{notations}

\begin{se} By expanding in the given orthonormal basis of real eigenfunctions, we get
\bea \label{zetaexpansion}
\zeta_{Y,a_0}(s) =  \sum_{\Psi_Y} \langle \Psi_Y|a_0 \Delta_Y^{-s} |\Psi_Y \rangle = \sum_{\lambda \in \Lambda_Y-\{0\}} \lambda^{-s} \sum_{\Psi \vdash \lambda} \int_Y a_0 \Psi^2 d\mu_Y.
\eea
\end{se}

In the ``commutative'' case considered here, one can express the two-variable zeta function in terms of the one-variable version, as follows: 

\begin{lemma} \label{2to1}
$\zeta_{Y,a_1,a_2}(s):=\zeta_{Y,g_Y(da_1,da_2)}(s).$
\end{lemma} 

\begin{proof}
We expand in a chosen basis of eigenfunctions: 
\begin{eqnarray} \label{goodold}
\tr(a_1[\Delta_Y,a_2]\Delta_Y^{-s}) &=& \sum_{\Psi_Y} \langle \Psi_Y | a_1[\Delta_Y,a_2]\Delta_Y^{-s} | \Psi_Y \rangle \\ 
&=& \sum_{\lambda \neq 0} \lambda^{-s} \sum_{\Psi_Y \vdash \lambda} \int_Y \left( \Psi_Y a_1 \Delta_Y(a_2 \Psi_Y) - a_1 a_2 \lambda(\Psi_Y)^2 \right) \, d\mu_Y, \nn
\end{eqnarray}
and use the product rule for the Laplacian to find 
\begin{eqnarray} \label{abcd}
& & \tr(a_1[\Delta_Y,a_2]\Delta_Y^{-s}) \\ & & =\sum_{\lambda \neq 0} \lambda^{-s} \sum_{\Psi_Y \vdash \lambda} \int_Y \left( \Psi_Y a_1 \Delta_Y(a_2) \Psi_Y - 2 a_1 g_Y(da_2,d\Psi_Y) \Psi_Y \right)\, d\mu_Y.\nn \end{eqnarray}
Now apply the divergence theorem to simplify the first summand in the integral: 
\bea 
\int_Y \Psi^2_Y a_1 \Delta_Y(a_2) \, d\mu_Y &=&  \int_Y g_Y(d(a_1 \Psi^2_Y),da_2)\, d\mu_Y \nn \\ &=&  \int_Y \Psi_Y^2 g_Y(da_1,da_2) + 2 \int a_1 \Psi_Y g_Y(d\Psi_Y,da_2)\, d\mu_Y. \nn
\eea
So we finally get 
$$ \tr(a_1[\Delta_Y,a_2]\Delta_Y^{-s}) = \sum_{\lambda \neq 0} \lambda^{-s} \sum_{\Psi_Y \vdash \lambda}  \int_Y \Psi_Y^2 g_Y(da_1,da_2)\, d\mu_Y = \tr(g(da_1,da_2) \Delta_Y^{-s}). $$
\end{proof}

\begin{lemma}\label{higs} The series $\zeta_{X,a_0}$ and $\zeta_{X,a_1,a_2}$ converge for $\Re(s)>\frac{\dim(X)}{2}$ and can be extended  to a meromorphic function on $\C$ with at most simple poles at $\frac12(\dim(X)-\Z_{\geqslant  0})$. \qed
\end{lemma}

\begin{proof}
See Higson \cite{Higsonzeta}, Thm.\ 2.1 for a more general statement that for a smooth linear partial differential operator $D$ of order $q$ on $X$, $\tr(D \Delta^{-s})$ has at most simple poles at $\frac12(\dim(X)+q-\Z_{\geqslant  0})$. For $\zeta_{X,a_0}$, we have $q=0$ and the statement follows; for $\zeta_{X,a_1,a_2}$, we have $q=1$, but from the previous lemma it follows that there is no pole at $\frac12(\dim(X)+1)$.
\end{proof}

\begin{lemma} \label{diffinv} The zeta-functions $\zeta_{X,a_0}$ and $\zeta_{X,a_1,a_2}$ are diffeomorphism invariants, in the sense that if $\varphi \, : X \rightarrow X$ is a smooth diffeomorphism, then 
$$ \zeta_{X,g_X,a_0} = \zeta_{X,\varphi^*(g_X),\varphi^*(a_0)}, $$
and$$ \zeta_{X,g_X,a_1,a_2} = \zeta_{X,\varphi^*(g_X),\varphi^*(a_1), \varphi^*(a_2)}. $$  \end{lemma}
\begin{proof} 
The map $\varphi$ is a Riemannian isometry $(X,g_X) \rightarrow (X,\varphi^*(g_X))$ and hence $$\varphi^{*} \Delta_{X,g_X} = \Delta_{X,g_X*} \varphi^{*}$$ and $\varphi$ preserves integrals, i.e., $$\int f d\mu_{g_X} = \int f^{*} d\mu_{g_X^{*}}$$ (cf.\ also Lemma \ref{Watson}). Hence $\varphi^{*}$ sends normalized eigenfunctions to normalized eigenfunctions with the same eigenvalue, so that we have:
\beastar
\zeta_{X,g_{X},a_{0}}(s) & = & \sum_{\lambda \neq 0} \lambda^{-s} \sum_{\Psi \vdash \lambda} \langle \Psi | a_{0} | \Psi \rangle_{g_{X}} \\
& = & \sum_{\lambda \neq 0} \lambda^{-s} \sum_{\Psi \vdash \lambda}  \langle \Psi^{*} | a_{0}^{*} | \Psi^{*} \rangle_{g^{*}_{X}} \\
& = & \zeta_{X,g^{*}_{X},a_0^*}(s).
\eeastar
For the 2-variable version, note that by assumption $g(da_{1},da_{2})^{*} =g(da_{1}^{*},da_{2}^{*})$ (see \cite{Watson}) and hence the invariance follows from that of the one-variable version by Lemma \ref{2to1}.
\end{proof}

\begin{remarks}

\mbox{}

- Observe that in conditions (a) and (b) in the main theorem, we only pull back the functions $a_i$, not the Riemannian structure with corresponding Laplace operator, so the identities in (a) and (b) are in general non-void. For diffeomorphism invariance in Lemma \ref{diffinv}, however, we pull back all structure, including the Laplace operator. 

- The spectrum (viz., $\zeta_X(s)$) is an \emph{incomplete} invariant of a Riemannian manifold: this is the problem of isospectrality. Connes (\cite{ConnesInv}) described a complete diffeomorphism invariant of a Riemannian manifold by adding to the spectrum the ``relative spectrum'' (viz., the relative position of two von Neumann algebras in Hilbert space). In another direction, B\'erard, Besson and Gallot (\cite{BBG}) gave a faithful embedding of Riemannian manifolds into $\ell^2(\Z)$, but by ``wave functions'', which are not diffeomorphism invariant. The family of zeta functions introduced here is some kind of diffeomorphism invariant when the $C^{\infty}$-diffeomorphism type of the manifold is fixed (viz., the algebras of functions $C^{\infty}(X)$ is given): these algebras of functions are used as ``labels'' for the zeta functions.  We do not know whether the \emph{sets} of functions $\zeta_{X,a_0}, \tzeta_{X,a_1}$ (without an explicit labeling) determine the isometry type of the manifold. 

- Using eigenvalues (viz, $\zeta_{X}(s)$) as dynamical variables in gravity was brought up by Landi and Rovelli (\cite{LR1}, \cite{LR2}). It would be interesting to adapt their theory by using  \emph{all} zeta functions.

\end{remarks}
 
 \section{A residue computation---Proof of Theorem \ref{Z0thm}} \label{mmm}
 
 The fact that (ii) implies (i) is easy, using the following lemma (see \cite{Watson}): 

\begin{lemma} \label{Watson}
Suppose that $\varphi : X \rightarrow Y$ is a smooth diffeomorphism of closed Riemannian manifolds. Let $$U=\varphi^* \; : \; L^2(Y) \rightarrow L^2(X)$$ denote the induced pullback map. Then $\varphi$ is a Riemannian isometry if and only if $U$ is a unitary operator that intertwines the Laplace operators on smooth functions: $\Delta_XU =U \Delta_Y$. \qed
\end{lemma}

\begin{remark}
If we don't assume that $U$ arises as actual pullback from a map, then the existence of such a $U$ merely implies that $X$ and $Y$ are isospectral, cf.\ also the discussion in \cite{ZelditchCFO}. 
\end{remark}

\begin{se}[Proof of (ii) $\Rightarrow$ (i) in Theorem \ref{Z0thm}] Pull-back by $\varphi$ induces a unitary transformation $U$ between $L^2(Y)$ and $L^2(X)$ that intertwines the respective Laplace operators.  From this intertwining, we find that for every $\lambda$, $U \Psi_{Y,\lambda}$ is a normalized eigenfunction of eigenvalue $\lambda$. From (\ref{zetaexpansion}), we get that $\zeta_{Y,a_0}(s)=\zeta_{X,a_0^*}(s)$ for all functions $a \in C^{\infty}(Y)$, and similarly for the two-variable version (cf.\ proof of Lemma \ref{diffinv}). \qed
\end{se} 

\bigskip 

For the other, more interesting direction of the proof, we first present a short and formal argument, by computing suitable residues of the zeta functions. In later sections, we will also compare expansion coefficients in the region of absolute convergence, rather than residues. This is computationally convenient, and it will allow us to prove some of the ``harder'' statements in Theorem \ref{extraprop}. 
 
 \begin{notation}
If $\varphi \, : \, X \rightarrow Y$ is a smooth diffeomorphism of Riemannian manifolds, we denote by $w_\varphi$ the change of the volume element by the map $\varphi$ (Radon-Nikodym derivative), i.e., locally in a chart, 
$$ w_\varphi = |\det(J_\varphi)| \sqrt{\det(g_Y)/\det(g_X)}, $$
where $J_\varphi$ is the jacobian matrix of $\varphi$ (sometimes, $w_\varphi$ is called the jacobian of $\varphi$). We then have the change of variables formula
\bea \label{chvar}
\int_{Y}  a_0 d\mu_Y & = & \int_{X} a_0^* w_\varphi d\mu_X,
\eea
for any function $a_0 \in C^{\infty}(Y)$. 
\end{notation}
 
 \begin{lemma}[e.g., \cite{Gilkey}, Lemma 1.3.7 and Thm.\ 3.3.1(1)] \label{residue}
Let $X$ denote a closed $d$-dimensional Riemannian manifold, $d>0$. Then for $a_{0} \in C^{\infty}(Y)$ the function $\Gamma(s) \zeta_{X,a_0}(s)$ has a simple pole at $d/2$  with residue 
\bea
\Res_{s=\frac{d}{2}} \zeta_{X,a_0} = \frac{1}{\Gamma(\frac d2)(4 \pi)^{d/2}} \int_X a_0\, d\mu_X. \qed \nn
\eea
\end{lemma}
 
 \begin{lemma} \label{res2} Let $X$ be a closed $d$-dimensional Riemannian manifold, $d>0$. For any $a_1,a_2 \in C^{\infty}(X)$ we have
 $$ \mbox{Res}_{s=d/2} \zeta_{X,a_1,a_2} = \frac{1}{\Gamma(\frac d2)(4 \pi)^{d/2}}  \int_X g_X(da_1,da_2)\, d\mu_X. $$
 \end{lemma}

\begin{proof}
Follows from the previous Lemma and Lemma \ref{2to1}. 
\end{proof}

\begin{proof}[First proof of Theorem \ref{Z0thm} \textup{((i)} $\Rightarrow$ \textup{(ii))}]

\begin{lemma} \label{basechange} The map $\varphi$ has $w_\varphi=1$.\end{lemma}

 \begin{proof}
It follows from $\zeta_{Y,a_0} = \zeta_{X,a_0^*}$  by taking residues that
$$ \Res_{s=\frac{d}{2}} \zeta_{Y,a_0}(s) = \Res_{s=\frac{d}{2}} \zeta_{X,a_0^*}(s). $$
At $a_0=1$, we find that $X$ and $Y$ have the same volume, and then by Lemma \ref{residue}, the general equality of residues  becomes
\bea \int_Y a_0 \, d\mu_Y = \int_X a_0^*\, d\mu_X. \eea The change of variables formula (\ref{chvar}) implies 
\bea 
\int_X a_0^* (1-w_\varphi)\, d\mu_X = 0 \ \ \ \  (\forall a_0^* \in C^{\infty}(X)). \nn
\eea
and the fundamental lemma of the calculus of variations gives that  \bea w_\varphi=1. \eea
\end{proof}

By using the polarisation identity for the quadratic form $g$, we see that 
\begin{eqnarray*} 4\zeta_{X,a^*_1,a^*_2} &=& 4\zeta_{X,g_X(da^*_1,da^*_2)} \\ &=&\zeta_{X,g_X(d(a^*_1+a^*_2),d(a^*_1+da^*_2))}-\zeta_{X,g_X(d(a^*_1-a^*_2),d(a^*_1-da^*_2))}\\ &=&\tzeta_{X,(a_1+a_2)^*} - \tzeta_{X,(a_1-a_2)^*} \\ &=& \tzeta_{Y,(a_1+a_2)} - \tzeta_{Y,(a_1-a_2)} \\ &=& 4 \zeta_{Y,a_1,a_2} \end{eqnarray*}
for all $a_1,a_2 \in C^{\infty}(Y)$. 

From Lemma \ref{res2}, we then get
$$\int_Y g_Y(da_1,da_2)\, d\mu_Y = \int_X g_X(da^*_1,da^*_2)\, d\mu_X. $$
After base change (using $w_\varphi=1$), the previous equality of integrals gives
\bea \label{mw} \int_X a_{1}^{*} \left( (\Delta_{Y} (a_{2}))^{*} - \Delta_{X}(a_{2}^{*})) \right) d\mu_X  =  \int_X a_{1}^{*} \left( \Delta_{Y}^{*} - \Delta_{X} \right)( a_{2}^{*} )d\mu_X \eea 
for all $a_1,a_2 \in C^{\infty}(Y)$.
Here, $\Delta_{Y}^{*} = U \Delta_{Y} U^{*}$ with $$U=\varphi^* \, : \, L^{2}(Y) \rightarrow L^{2}(X)$$ the pullback,  and $U^{*}$ the push-forward. Since this holds for all $a_1^*$, we find that 
$$ \Delta_{Y}^{*} = \Delta_{X}. $$
Also, $U$ is unitary, since $w_\varphi=1$, whence $$\langle Uf, Ug \rangle_X = \int_X f^* g^* w_\varphi d\mu_X = \int_Y fg d\mu_Y = \langle f, g \rangle_Y $$ for all $f,g \in L^{2}(Y)$.
Hence from Lemma \ref{Watson}, we find that $\varphi$ is an isometry. 
\end{proof}

\section{Matching squared eigenfunctions} \label{sq}

In this section, we investigate more closely the meaning of condition (a) in the main theorem. 

\begin{notation}
We use the following notation for the sum of the squares of the eigenfunctions belonging to a fixed eigenvalue and basis:
$$\sigma_{X,\lambda}=\sigma_\lambda:=\sum_{\Psi_X \vdash \lambda} \Psi_X^2 . $$
\end{notation}

\begin{proposition} \label{squares0}
Suppose that $\varphi \, : \, X \rightarrow Y$ is a smooth diffeomorphism between connected closed Riemannian manifolds. Let $\{\Psi_{X,\lambda}\}$ and $\{\Psi_{Y,\mu}\}$ denote two complete sets of orthonormal real eigenfunctions for $\Delta_X$ and $\Delta_Y$, respectively. Condition (a) in Theorem \ref{Z0thm} is equivalent to the statement that the spectra of $\Delta_X$ and $\Delta_Y$ agree with multiplicities, we have $w_\varphi=1$, and for any eigenvalue $\lambda$ we have $$ \sigma_{X,\lambda} := \sum_{\Psi_X \vdash \lambda} (\Psi_X)^2=\sum_{\Psi_Y \vdash \lambda} (\Psi^*_Y)^2  =: \sigma_{Y,\lambda}^*.$$
\end{proposition}

\begin{proof}
The assumption is $$\zeta_{Y,a_0}(s)=\zeta_{X,a_0^*}(s)$$ for all $a_0 \in C^{\infty}(Y)$. 
Evaluated at the unit $a_0=1$, it follows that the nonzero spectra of $\Delta_X$ and $\Delta_Y$, including multiplicities, agree (using the identity principle for generalized Dirichlet series, cf. \cite{HardyDirichlet}, Thm.\ 6), and  this implies that both the volumes and dimensions of $X$ and $Y$ agree as well.

The coefficients of the above Dirichlet series (when grouped according to fixed $\lambda$) as in (\ref{zetaexpansion}) are uniquely determined by it, again by the identity theorem for Dirichlet series. If we spell out the assumption for the individual coefficients in $\lambda^s$ in the two Dirichlet series, we find  that  for any $a_0 \in C^{\infty}(Y)$ we have $$\int_{Y} \sum_{\Psi_Y \vdash \lambda} |\Psi_Y|^2 a_0
  d\mu_Y =  \int_{X} \sum_{\Psi_X \vdash \lambda} |\Psi_X|^2 a_0^*  d\mu_X $$
for $\lambda \neq 0$. We perform a coordinate change in the first integral by using the map $\varphi : X \rightarrow Y$. Since $w_\varphi=1$, we find 
\bea
\int_{X} \sum_{\Psi_Y \vdash \lambda} |\Psi^*_Y|^2 a_0^* d\mu_X & = & \int_{X} \sum_{\Psi_X \vdash \lambda} |\Psi_X|^2 a_0^*  d\mu_X
\eea
 for any $a_0 \in C^{\infty}(Y)$. Again, the fundamental lemma of the calculus of variations gives 
\bea \label{eqeig}
\sum_{\Psi_X \vdash \lambda} (\Psi_X)^2=  \sum_{\Psi_Y \vdash \lambda} (\Psi^*_Y)^2, 
\eea
for $\lambda \neq 0$. The eigenvalue $\lambda=0$ has multiplicity one, since the manifold is connected, and the normalized eigenfunction on $Y$ is equal to $1/\sqrt{\vol(Y)}$, which pulls back to $1/\sqrt{\vol(Y)}$. Since $X$ and $Y$ have the same volume (from equality of their zeta functions at $a_0=1$), we find this is equal to $1/\sqrt{\vol(X)}$, the normalized eigenfunction for $\lambda=0$ on $X$. 

The other direction of the equivalence is obtained by reversing steps. This finishes the proof. \end{proof}

We deduce a corollary about the diagonal of the heat kernel:

\begin{corollary}
If $\varphi : X \rightarrow Y$ is a smooth diffeomorphism of closed connected smooth Riemannian manifolds, then the following conditions are equivalent:
\begin{enumerate}
\item[\textup{(A)}] For all $a_0 \in C^{\infty}(Y)$, we have that $ \zeta_{Y,a_0} = \zeta_{X,\varphi^*(a_0)}$; 
\item[\textup{(B)}] $K_X(t,x,x) = K_Y(t,\varphi(x),\varphi(x))$ for all $t>0$ and all $x \in X$, and $w_\varphi=1$. 
\end{enumerate} 
\end{corollary}

\begin{proof}
Recall the following expression for the heat kernel (e.g., \cite{Berard}, V.3)
\bea \label{kernelxy}
K_X(t,x,y) = \mathop{\sum_{\lambda \in \Lambda}}_{\textrm{distinct}}  e^{-\lambda t} \sum_{\Psi \vdash \lambda} \Psi(x){\Psi(y)}, \, \, (t>0)
\eea
and setting $x=y$, we find 
\bea \label{kernel}
K_X(t,x,x) = \mathop{\sum_{\lambda \in \Lambda}}_{\textrm{distinct}}  e^{-\lambda t} \sigma_{X,\lambda}(x), \, \, (t>0).
\eea
This implies the result. \end{proof}

\section{Expansion coefficients of the two-variable zeta function} \label{arbspec}

We now take a closer look at the expansion coefficients of the two-variable zeta functions, under the assumptions of (i) in Theorem \ref{Z0thm}. This computation provides an alternative proof of Theorem \ref{Z0thm}, and will be used in proving part of Theorem \ref{extraprop}. 

We find that $X$ and $Y$ have the same spectra with multiplicities. As usual, we denote this spectrum by $\{ \lambda \}$ (with multiplicities). We have already seen that the polarisation identity for the quadratic form $g$ implies that $\zeta_{Y,a_1,a_2} = \zeta_{X,a_1^*,a_2^*}$. Our starting point is the expression for $\tr(a_1[\Delta,a_2]\Delta_Y^{-s})$ from Equation (\ref{abcd}). The coefficient in $\lambda^{-s}$ is 
$$ \int_Y a_1 \left[ \sigma_{Y,\lambda} \Delta_Y(a_2) - 2 g_Y (da_2,d\sigma_{Y,\lambda}) \right] d\mu_Y. $$
If we equate this to the corresponding coefficient of the other zeta function, and then perform a base change to $X$ (using $w_\varphi=1$) and use the fundamental lemma of calculus of variations to remove the integral over $X$, we find 
\bea  \sigma_\lambda (\Delta_Y^*-\Delta_X) &=& 2 g_Y^*(d-,d\sigma_{\lambda}) - 2 g_X(d-,d\sigma_{\lambda}) \label{used}  \\ &=& \mbox{ first order operator. } \nn  \eea
Equation (\ref{used}) means that the leading symbol of $\Delta_Y^*-\Delta_X$ vanishes outside the zero set of $\sigma_\lambda$,  which (by \cite{Watson}) implies that $g_Y^*=g_X$. Since for every $x$ there is a $\lambda$ with $\sigma_\lambda(x) \neq 0$, we find $g_Y^*=g_X$ everywhere. Hence $\varphi$ is an isometry. 

\section{Improvements in the case of simple spectrum}

In this section, we consider how to improve the theorem in case the spectrum of $\Delta_X$ is simple; we will prove Theorem \ref{extraprop}. We first start by listing some consequences of known results related to condition (a): 

\begin{remarks}

\mbox{}

- Condition (a) in Theorem \ref{Z0thm} does not {always} suffice to imply that $\varphi$ is an isometry, cf.\ Corollary \ref{exist}. 

- There exist isospectral, non-isometric compact Riemannian manifolds with simple spectrum (cf.\ Zelditch \cite{Zelditchmult}, Theorem C), so (for maps with unit jacobian) condition (a) is not equivalent to isospectrality (which would be condition (a) only for the identity function). 

- A result of Uhlenbeck (\cite{Uh1}) says that the condition of having non-simple Laplace spectrum is meager in the space of smooth Riemannian metrics on a given manifold $X$.  Thus, Theorem \ref{extraprop} treats the `generic' situation. But there do exist Riemannian manifolds for which the multiplicity of the spectrum grows polynomially in the eigenvalues, cf.\ e.g.\ Donnelly \cite{Don}. 
\end{remarks}

\begin{lemma} \label{squares}
Suppose that $X$ is a closed smooth Riemannian manifold. Then the zero set of any nonzero eigenfunction of $\Delta_X$ is not dense. If we let $\tilde{X} \subseteq X$ denote complement of the union of all such zero sets, then $\tilde{X}$ is dense in $X$, and the following holds: for any real $\Delta_X$-eigenfunction $\Phi$, and any function $h \in C(X)$ that satisfies $h^2=\Phi^2$, we have that $h=\pm \Phi$ on every connected component of $\tilde{X}$.   
\end{lemma}

\begin{proof} 
We can write $c \Phi =h$ where $c$ is a function (a priori not necessarily globally constant) that takes values in $\{+1,-1\}$. We can assume that $X$ is connected. We choose for $\tilde{X}$ the complement of the union of all zero sets of non-zero $\Delta_X$-eigenfunctions on $X$. By continuity, $c$ is obviously constant on connected components of $\tilde{X}$. All we have to show is that $\tilde{X}$ is dense. 

We claim that the complement of the zero set of an eigenfunction $\Phi$ is an open dense subset of $X$. Granting this for the moment, since the spectrum is discrete, the intersection  $\tilde{X}$ of all such complements of zero sets is a countable intersection of open dense subsets of $X$. Since $X$ is compact and hence a complete metric space (for the Riemannian metric), the Baire category theorem implies that this intersection is itself dense. 

The claim will follow if we show that  the zero set $\cZ$ of $\Phi$ is nowhere dense. So suppose on the contrary that $\cZ$ is dense in a neighbourhood $U$ of some point $x \in X$. Then $\Phi \equiv 0$ on $\overline{U}$. Since the unique continuation theorem applies to the Laplacian with smooth coefficients (cf.\ \cite{AKS}, Rmk.\ 3, p.\ 449) we find $\Phi \equiv 0$ on all of $X$ (assumed connected), a contradiction. 
\end{proof}

\begin{proof}[Proof of Theorem \ref{extraprop}]

First, suppose $\varphi$ is an isometry between $X$ and $Y$. Pull-back by $\varphi$ induces a unitary transformation $U$ between $L^2(Y)$ and $L^2(X)$ that intertwines the respective Laplace operators.  From this intertwining, we find that for every $\lambda$, $U \Psi_{Y,\lambda}$ is a normalized eigenfunction of eigenvalue $\lambda$, hence equal to $\pm \Psi_{X,\lambda}$ by the simplicity assumption on the spectrum.  From (\ref{zetaexpansion}), we get that $\zeta_{Y,a_0}(s)=\zeta_{X,a_0^*}(s)$ for all functions $a_0 \in C^{\infty}(Y)$. 

For the more interesting converse direction, we know from the previous section that the pullback map $U=\varphi^*$ takes on the form
\beastar
U : L^ {2}(Y) & \to & L^{2}(X) \\
\Psi_{Y,\lambda} (\lambda \neq 0) & \mapsto & \varphi^{*}(\Psi_{Y,\lambda}) = c_\lambda \Psi_{X,\lambda}; \ \ c_\lambda \in \{ \pm 1\}\\
\Psi_{Y,0}=\frac{1_Y}{\sqrt{\vol(Y)}} & \mapsto & \varphi^*(\Psi_{Y,0}) = \frac{1_X}{\sqrt{\vol(Y)}}=\frac{1_X}{\sqrt{\vol(X)}}=\Psi_{X,0}
\eeastar
The map is clearly unitary and bijective. We prove that the map $U$ also intertwines the Laplace-Beltrami operators. For this, let $\tilde{X}_{\lambda}$ denote the zero set of the eigenfunction $\Psi_{X,\lambda}$. Let $x \in \tilde{X}_{\lambda}$. We can find an open neighbourhood $\mathcal{U}_x$ of $x$ on which $c_\lambda=\pm 1$ (defined by  $\varphi^{*}(\Psi_{Y,\lambda}) = c_\lambda \Psi_{X,\lambda}$ as above) is constant. For any $\tilde{x} \in \mathcal{U}_x$, we find 
$$ \Delta_X U \Psi_{Y,\lambda}(\tilde{x}) =  \Delta_X (c_\lambda \Psi_{X,\lambda})(\tilde{x}) = c_\lambda \lambda \Psi_{X,\lambda}(\tilde{x}) = U(\lambda \Psi_{Y,\lambda})(\tilde{x}) = U \Delta_Y \Psi_{Y,\lambda} (\tilde{x}).$$
This equality of continuous functions (for the continuity of the left hand side, use that the map $\varphi$ is assumed to be smooth)  holds on $\tilde{X}_{\lambda}$, and since $\tilde{X}_{\lambda}$ is dense in $X$ (see the previous proof), we find that it holds on $X$. Now since the eigenfunctions form a basis for $L^2(X)$, we find an equality of operators $$\Delta_X U = U \Delta_Y.$$

This implies that $X$ and $Y$ are isometric by the previous lemma, and finishes the proof of the second part of Theorem \ref{extraprop}. 

For the first part, we observe that the zero set of $\sigma_\lambda$ is nowhere dense (since $\sigma_\lambda$ is a finite linear combination of the positive functions $\Psi^2$ for $\Psi \vdash \lambda$, so $\sigma_\lambda=0$ implies $\Psi=0$ for all $\Psi \vdash \lambda$, and use Lemma \ref{squares}). Hence in this case, in the proof of Theorem \ref{Z0thm}, it suffices to have formula (\ref{used}) for \emph{only one $\lambda$}, i.e., equality of \emph{one} coefficient of the Dirichlet series in condition (b) suffices. 
\end{proof}

\section{Further improvements}

At the cose of using more ``hard'' analysis, we can improve some of the auxiliary results from the previous section even further. 

\begin{lemma}
If in Lemma \ref{squares} we assume that $h \in C^{\infty}(X)$, then $h=\pm \Phi$ with the sign constant everywhere.
\end{lemma}

\begin{center}
\begin{figure}[h]
\includegraphics[width=8cm]{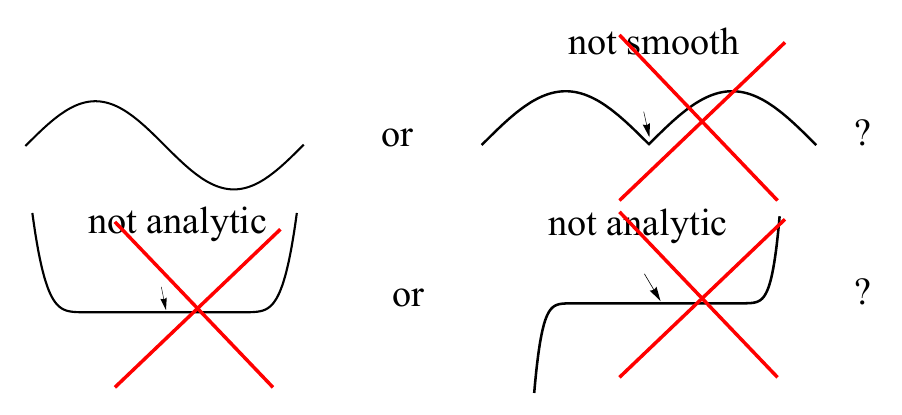}
\caption{An eigenfunction around a nodal set}
\label{fig0}
\end{figure}
\end{center}

\begin{remark} 

We can write $c \Phi =h$ where $c$ is a function (a priori not necessarily globally constant) that takes values in $\{+1,-1\}$. We have to prove that $c$ is globally constant.

The gist of the proof is to use the regularity of eigenfunctions at zeros. Think of the prototypical $\sin(x)=c(x)h(x)$ on $[0,2\pi]$. If $h(x)$ is \emph{not} equal to $\pm \sin(x)$, then we have $h(x)=|\sin(x)|$. But that function is not even $C^1$ at $x=\pi$. On the other hand, functions like
$ f(x)^2 =e^{-2/|x|} (x \neq 0); \ f(0)=0 $
have different smooth $f(x)$ as square root, but have a zero of infinite order. See Figure \ref{fig0}. 
\end{remark} 

\begin{proof}
It follows e.g.\ from the analysis in Caffarelli and Friedman (\cite{CF}, Example 3 pp.\ 432--433, compare: \cite{HL}, chapter 4, proof of Lemma 4.1.1) that for every point $x_0$ of the manifold $X$, there exists a small enough neighbourhood $U$ of $x_0$ that intersects the zero set $\cZ$ of $\Phi$ in the union of finitely many submanifolds of dimension $\leqslant m-1$. First note that if $x_{0} \notin \cZ$ there exists an open set $W \ni x_{0}$ for which $\Phi\mid_{W} \neq 0$. Then the function $$\left( \frac{h}{\Phi} \right) \mid_{W} = c\mid_{W}$$ is smooth and hence $c \mid_{W}$ must be constant.

\begin{center}
\begin{figure}[h]
\includegraphics[width=4cm]{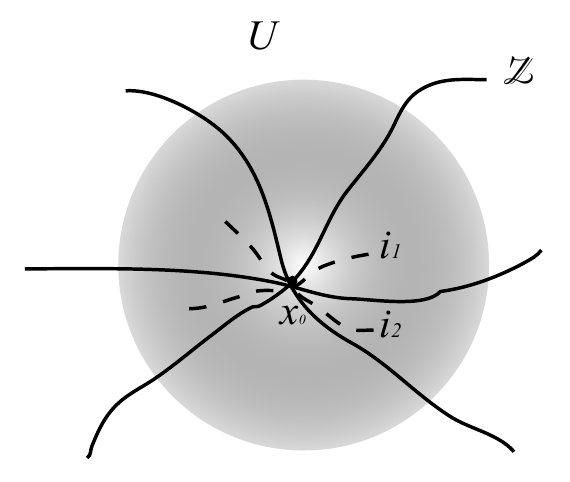}
\caption{Local structure of the zero set $\mathcal{Z}$ in a neighborhood $U$ of $x_0$, with independent paths $i_1,i_2$}
\label{fig01}
\end{figure}
\end{center}

Now, let $x_{0} \in \cZ$, then both $h$ and $\Phi$ vanish at $x_{0}$. Choose any $C^{\infty}$ path $i: ]-\varepsilon,\varepsilon[ \to X$ such that $$\im(i) \cap \cZ=\{x_0\} \mbox{ with }i(0)=x_0 \mbox{ and }||i'(t)||>0.$$ Then we get:
\bea
h(i (t)))=c( i (t))\cdot \Phi(i(t))
\eea
Assume that $c( i (t))$ changes sign at $0$.
Differentiating the equation at $t=0$ gives:
$$
\lim_{t \downarrow 0} h(i(t))'  =  \lim_{t \downarrow 0} \Phi(i(t))'   =  - \lim_{t \uparrow 0} \Phi(i(t))'. $$
Because of smoothness, this implies $$(\Phi \circ i)'(0) = (h\circ i)'(0)=0.$$
Any path in $\cZ$ is mapped identically to $0$ by both $h$ and $\Phi$ and hence also has derivative zero. It follows that both $h$ and $\Phi$ have all (directional) derivatives in $x_{0}$ equal to $0$. Indeed, because locally around $x_0$ the zero set $\cZ$ is contained in a finite union of codimension $\geqslant 1$ submanifolds, we can find $m$ paths $i$ as above whose tangent vectors at $x_0$ span $T_{x_0} X$, and since all directional derivatives along these vectors are zero, so is the total derivative. \\
By induction it follows that up to any order all derivatives vanish. Hence $x_{0}$ is a zero of the eigenfunction $\Phi$ of infinite order, which is impossible by Aronszajn's unique continuation theorem. We conclude that locally $c$ does not change sign at zeros (and anyhow not at nonzeros). We assume $X$ to be connected, so this implies that $c$ is constant.
\end{proof}

We deduce the following corollary:

\begin{corollary} \label{verysmooth}
If $\varphi : X \rightarrow Y$ is a $C^{\infty}$-diffeomorphism of closed connected $C^\infty$-Riemannian manifolds and with simple Laplace spectrum, such that the diagonals of the heat kernels match up in the sense that $K_X(t,x,x)=K_Y(t,\varphi(x),\varphi(x))$ for sufficiently small $t>0$, then the heat kernels match up in the sense that $K_X(t,x,y)=K_Y(t,\varphi(x),\varphi(y))$ for all $t>0$. 

In particular, if $g$ and $g'$ are two smooth Riemannian structures on a closed connected manifold and with simple Laplace spectrum, then $K_g(t,x,x)=K_{g'}(t,x,x)$ for sufficiently small $t>0$ implies that $K_g(t,x,y)=K_{g'}(t,x,y)$ for all $t>0$ and hence $g=g'$. 
\end{corollary}

\begin{proof} Since $$K_X(t,x,x)=K_Y(t,\varphi(x),\varphi(x))$$ for sufficiently small $t>0$, and we have simple Laplace spectrum, formula (\ref{kernel}) implies that the corresponding spectra, and the squares of eigenfunctions match up: $$(\Psi_{X,\lambda})^2=(\varphi^*\Psi_{Y,\lambda})^2.$$ The smoothness on both sides implies via the previous lemma that the functions agree up to a global sign. Applying (\ref{kernelxy}), we find
\begin{eqnarray*} K_X(t,x,y) &=& \sum_\lambda  e^{-\lambda t} \Psi_{X,\lambda}(x){\Psi_{X,\lambda}(y)} \\
&=&  \sum_\lambda  e^{-\lambda t} (\pm 1)^2\varphi^*\Psi_{Y,\lambda}(x){\varphi^*\Psi_{Y,\lambda}(y)} \\ &=& K_Y(t,\varphi(x),\varphi(y)). 
\end{eqnarray*}
The particular case follows by setting $\varphi$ to be the identity map. 
\end{proof}

\section{Example: flat tori} \label{tori}

\begin{se} Let $\T=\R^n/\Lambda$ denote a flat torus, corresponding to a lattice $\Lambda$ in $\R^n$. Let $\Lambda^\vee$ denote the dual lattice to $\Lambda$. The Laplacian is $$\Delta_\T= -\sum_k \partial^2_k,$$ the spectrum is $$\{ 4 \pi^2 ||\lambda^\vee||^2\}_{\lambda^\vee \in \Lambda^\vee},$$ a basis of orthogonal eigenfunctions of eigenvalue $\ell$  is given by $$\Psi_{\lambda^\vee}:=\frac{e^{2 \pi i\langle \lambda^\vee,x \rangle}}{\sqrt{\vol(T)}}$$ if $||\lambda^\vee||^2=\ell.$ (This is not a real basis as usual in this paper, but we will make appropriate adaptations.) The crucial property for us is that these functions satisfy $$|\Psi_{\lambda^\vee}|^2=\Psi_{\lambda^\vee} \cdot \overline{\Psi_{\lambda^\vee}}= \frac{1}{\vol(T)}.$$
\end{se}

\begin{se}
We consider the meaning of condition (a) in Theorem \ref{Z0thm} for the torus $\T$. Let $a_0 \in C^{\infty}(\T)$. Then 
\begin{eqnarray} \label{torus1}  \zeta_{\T,a_0}(s) &=& \sum_{\lambda^\vee \in \Lambda^\vee} \frac{1}{||4 \pi^2 \lambda^\vee||^{2s}} \cdot \frac{1}{\vol(T)} \int_\T a_0 |\Psi_{\lambda^\vee}|^2 \, d\mu_{\R^n} \\ &=& \left( \frac{1}{\vol(T)} \int_\T a_0 \, d\mu_{\R^n}  \right) \cdot \zeta_\T(s). \nn \end{eqnarray} We conclude from this by noting that the volume is determined by the spectrum:
\end{se}

\begin{proposition}
Let $\varphi \, : \, \T_1\rightarrow \T_2$ denote a smooth diffeomorphism between two flat tori. Then the following are equivalent: 
\begin{enumerate}
\item[\textup{(i)}] For all $a_0 \in C^{\infty}(T_{2})$, we have that $ \zeta_{\T_2,a_0} = \zeta_{\T_1,\varphi^*(a_0)}$; 
\item[\textup{(ii)}] $\T_1$ and $\T_2$ are isospectral, and $\varphi$  has jacobian $w_\varphi= 1$. \qed
\end{enumerate} 
  \end{proposition}

\begin{corollary} \label{exist} There exist non-isometric manifolds for which condition (a) of Theorem \ref{Z0thm} holds.
\end{corollary}

\begin{proof} 
Take the following isospectral, non-isometric tori $\T_{\pm}$ (\cite{Schiemann}, \cite{CS}) in dimension 4, spanned by the column vectors in the respective matrices  
$G_+$ and $G_-$
$$  G_{\pm}=\frac{1}{2\sqrt{3}} \left(
\begin{array}{rrrr}
 \pm 3 & -\sqrt{7} & -\sqrt{13} & -\sqrt{19} \\
 1 & \pm 3 \sqrt{7} & \sqrt{13} & -\sqrt{19} \\
 1 & -\sqrt{7} & \pm 3 \sqrt{13} & \sqrt{19} \\
 1 & \sqrt{7} & \sqrt{13} & \pm 3 \sqrt{19}
\end{array}
\right). $$
Consider the linear map $ A \, : \, \R^4 \rightarrow \R^4$ given by 
$$ A= G_+ G_-^{-1} = \frac{1}{5} \left(
\begin{array}{rrrr}
 -3 & -2 & -1 & -3 \\
 2 & -2 & 4 & -3 \\
 3 & -3 & -4 & 3 \\
 1 & 4 & 2 & -4
\end{array}
\right) $$ 
with determinant $\det(A)=1$. This map factors through to a map $\T_- \rightarrow \T_+$ with determinant ($=$ jacobian) $1$.  \end{proof}

\begin{se}
We now consider condition (b) in Theorem \ref{Z0thm} for the torus $\T$. We compute for $a_1, a_2 \in C^{\infty}(\T)$, using Lemma \ref{2to1}, that \begin{equation} \label{torus2} \zeta_{\T,a_1,a_2}(s)= \left( \frac{1}{\vol(T)} \int_\T \nabla(a_1)^\top \nabla(a_2) \, d\mu_{\R^n}  \right) \zeta_\T(s).\end{equation} 
 \end{se}

\part{LENGTHS AND DISTANCES}

\section{Length categories} \label{lengthcat}

\begin{definition}
We call a pair $(\cC,\ell)$ a \emph{length category} if $\cC$ is a category endowed with a subcategory $\dD$, full on objects, such that every morphism in $\dD$ is an isomorphism in $\cC$ and $\dD$. These are called $\dD$-isomorphisms from now on. Furthermore, for every $X,Y \in \mathrm{Ob}(\cC)$ and every $\varphi \in \Hom(X,Y)$, there is defined a positive real number $\ell(\varphi) \in \R_{\geqslant  0}$, called the \emph{length of $\varphi$} such that 
\begin{enumerate}
\item[\textup{\textbf{(L1)}}] $\ell(\varphi)=0$ if and only if $\varphi$ is an $\dD$-isomorphism; 
\item[\textup{\textbf{(L2)}}] If $X,Y,Z \in \mathrm{Ob}(\cC)$ and $\varphi \in \Hom(X,Y), \psi \in \Hom(Y,Z)$, then 
$$ \ell(\psi \circ \varphi) \leqslant  \ell(\varphi) + \ell(\psi).$$
\end{enumerate}
\end{definition}

\begin{remark}
In particular, in $\textup{\textbf{(L1)}}$ we do not assume that the $\dD$-isomorphism classes are necessarily the categorical isomorphism classes (i.e, the maps for which there exists an inverse in the category $\cC$), but we do assume that the $\dD$-isomorphisms are (some of the) categorical isomorphisms of $\cC$. For instance, think of the category $\cC$ of metric spaces and continuous maps, but with $\dD$-isomorphisms the isometries (instead of the homeomorphisms). Note also that the morphisms of $\dD$ can be recovered from the pair $(\cC,\ell)$ as those morphisms in $\cC$ with length zero.
\end{remark}

To illustrate the concept, let us look at some examples. If there is no subcategory $\dD$ specified then implicitly it is understood that the $\dD$-isomorphisms are the $\cC$-isomorphisms.

\begin{examples}

\mbox{ }

- Any category is a length category in a trivial way, defining the ``discrete'' length by $$\ell(X \cong Y)=0 \mbox{ and }\ell(\varphi)=1 \mbox{ otherwise.} $$ However, categories can carry other, more meaningful lengths. 

- Let $\mathsf{Grp}$ denote the category of finite commutative groups, and for $\varphi \in \Hom(G,H)$ a homomorphism of groups $G$ and $H$, define its length as
$$\ell(\varphi) = \max \{ \log(|\ker(\varphi)|), \log(|\coker(\varphi)|) \}. $$ 
This obviously satisfies \textbf{(L1)} and also \textbf{(L2)} since $|\ker(\psi \circ \varphi)| \leq |\ker(\varphi)| \cdot |\ker(\psi)|$ and similarly for the cokernel. Hence $(\mathsf{Grp},\ell)$ is a length category. 

- More generally, an \emph{abelian} category with in some sense ``measurable'' kernels and cokernels is a length category by a similar construction. However, non-abelian categories can also be length categories for an interesting length function. In some sense, this is a metric substitute for the non-existence of kernels/cokernels. 

- The category of compact metric spaces with bi-Lipschitz homeomorphisms is a length category for the length 
$$ \ell(\varphi) := \max\{ \left|\log \dil(\varphi)\right|, \left|\log \dil(\varphi^{-1})\right| \}, $$
where $\dil(\varphi)$ is the dilatation of the map $\varphi$. This length induces Lipschitz distance between compact metric spaces.  \end{examples}

Lengths in categories sometimes give rise to a metric on the moduli space of objects of the category $\cC$ up to $\dD$-isomorphism, as the following lemma shows (note that the condition is sufficient, but not necessary):

\begin{lemma} \label{ddd}
If $(\cC,\ell)$ is a length category and we put $$d(X,Y)=\frac{1}{2} \left(\inf_{\varphi \in \Hom_{\cC}(X,Y)} \{ \ell(\varphi),+\infty\} \, + \, \inf_{\psi \in \Hom_{\cC}(Y,X)} \{\ell(\psi),+\infty\} \right) $$
then $d$ is an extended (i.e., $(\R \cup\{ +\infty \})$-valued) metric on the ``moduli space'' $\mathrm{Ob}(\cC)/\dD \mathrm{-iso}$ if for $d(X,Y)=0$, the infimum in the definition of $d$ is attained in $\Hom(X,Y)$. If $\Hom(X,Y) \neq \emptyset$ for any $X,Y \in \mathrm{Ob}(\cC)$, $d$ is a (finite) metric. 
\end{lemma}

\begin{proof}
First of all, length is well-defined on objects up to isomorphism: if $\varphi$ is arbitrary and $\psi$ is an isomorphism, then $$\ell(\varphi \circ \psi) \mathop{\leqslant }_{\mbox{\footnotesize{\textbf{(L2)}}}} \ell(\varphi) + \ell(\psi)  \mathop{=}_{\mbox{\footnotesize{\textbf{(L1)}}}} \ell(\varphi) =   \ell(\varphi \circ \psi \circ \psi^{-1})  \mathop{\leqslant }_{\mbox{\footnotesize{\textbf{(L2)}}}}\ell(\varphi \circ \psi) + \ell(\psi^{-1})  \mathop{=}_{\mbox{\footnotesize{\textbf{(L1)}}}} \ell(\varphi \circ \psi). $$
The positivity of $d$ is clear. For the triangle inequality, since $d$ is defined as the symmetrization of the hemimetric $$d'(X,Y)=\inf_{\varphi \in \Hom(X,Y)} \{\ell(\varphi),+\infty\}, $$ it suffices to prove the triangle inequality for $d'$. Let $\varepsilon>0$. Let 
$\varphi \in \Hom(X,Y)$ and $\psi \in \Hom(Y,Z)$ be such that $$\len(\varphi) \leqslant  d'(X,Y)+\varepsilon/2\mbox { and } \len(\psi) \leqslant  d'(Y,Z)+\varepsilon/2$$ (which is possible by the definition of length as an infimum). 
We have 
$$ d'(X,Z) = \inf_{\theta \in \Hom(X ,Z)} \len(\theta) \leqslant  \len(\psi \circ \varphi). $$ By axiom \textbf{(L2)} we find $$ \len(\psi \circ \varphi) 
 \leqslant     \len(\psi) + \len(\varphi)  \leqslant   d'(Y,Z)+d'(X,Y) + \varepsilon. $$
The triangle inequality follows by letting $\varepsilon$ tend to zero. 
Finally, assume $d(X,Y)=0$. Since the infimum in the definition is attained, we find a map $\varphi \in \Hom(X,Y)$ of length zero. Then axiom \textbf{(L1)} implies that $X\cong Y$. 
\end{proof}

\section{The length of a map between Riemannian manifolds} \label{length}

We will now consider the category $\cR$ of closed smooth Riemannian manifolds, with homomorphisms smooth diffeomorphisms and $\dD$-isomorphisms the isometries. We define a length function in this category using our diffeomorphism invariant. The idea is to measure how far the one- and two-variable zeta functions $\zeta_{Y,a_0}$ and $\tzeta_{Y,a_1}$  are apart under pullback by the map $\varphi$  in some suitable distance on the set of meromorphic functions, and to test this over certain well-behaved sets of test-functions $a_0,a_1$. 

\begin{definition}  
Let $f$ and $g$ denote two functions that are holomorphic non-zero in a right half line $$H_{\sigma}:=\{ s \in \R \mid s \geqslant \sigma \},$$ where $\sigma$ is fixed once and for all.  Define 
$$ \delta_1(f,g):= \mathop{\sup_{\sigma \leqslant s\leqslant \sigma+1}} \{ | \log \left| \frac{f(s)}{g(s)} \right| | \}, $$ and set $$d_\sigma(f,g):=  \frac{\delta_1(f,g)}{1+\delta_1(f,g)}.$$ 
\end{definition}

Convergence in $d_\sigma$ is \emph{not} uniform convergence of general analytic functions without zeros on $H_\sigma$ (because the absolute value signs cause an indeterminacy up to an analytic function with values in the unit circle), but when specialized to our Dirichlet series, this problem disappears, cf.\ infra. 

\begin{definition}
The \emph{length of a smooth diffeomorphism $\varphi\, : \, X \rightarrow Y$} of Riemannian manifolds of dimension $N$ is defined by 
$$ \ell(\varphi):= \mathop{\sup_{a_0 \in C^{\infty}(Y,\R_{\geqslant 0})-\{0\}}}_{a_1 \in C^{\infty}(Y)-\R} \! \! \! \max \, \{ d_{N}(\zeta_{X,a^*_0},\zeta_{Y,a_0}), d_{N}(\tzeta_{X,a^*_1}, \tzeta_{Y,a_1}) \}.$$
\end{definition}

\begin{remarks}

\mbox{ } 

- Since $d_\sigma$ is obviously bounded by $1$, the length of a map also takes values in $[0,1]$. 

- There is some arbitrariness in the definition of $\ell(\varphi)$: our zeta functions are holomorphic in $\Re(s) > \frac{N}{2}$, so one might also take the product metric of the suprema over an other suitable subset.
\end{remarks}

The main theorem can be rephrased as follows, which shows that $(\cR,\ell)$ satisfies axiom \textbf{(L1)} of a length category:

\begin{proposition}
If $X$ and $Y$ are closed Riemannian manifolds, then a smooth diffeomorphism $\varphi \, : \, X \rightarrow Y$ has length zero if and only if it is an isometry.
\end{proposition}

\begin{proof}
If $\varphi$ has length zero, then we have an equality of absolute values of zeta functions under pullback, at positive functions $a_0 \in C^{\infty}(Y,\R_{\geqslant 0})$ and functions $a_{1} \in C^{\infty}(Y)$. 

Since all eigenvalues are positive, and all Dirichlet series coefficients of the zeta functions are positive when evaluated at a positive function $a_0$ (cf.\ section \ref{sq}), the values for $s \in H_{d+1}$ of such zeta functions are positive, and hence equal. Now a standard theorem (e.g., \cite{Serre}, Section 2.2) implies that the two Dirichlet series are everywhere equal. We conclude that $\zeta_{X,a^*_0}=\zeta_{Y,a_0}$ for $a_0 \in C^{\infty}(Y,\R_{\geqslant 0})$ and $\tzeta_{X,a^*_1}(s) = \tzeta_{Y,a_1}$ for all $a_1 \in C^{\infty}(Y)$. Since any $a_0 \in C^{\infty}(Y)$ is a linear combination of such positive functions, we can apply Theorem 
\ref{Z0thm} to conclude that $\varphi$ is an isometry. The converse statement follows directly from the same theorem. 
\end{proof}

We now prove that $(\cR,\ell)$ also satisfies axiom \textbf{(L2)} of a length category: 

\begin{proposition}
If $X,Y,Z$ are closed Riemannian manifolds, and $\varphi \, : \, X \rightarrow Y$ and $\psi \, : \, Y \rightarrow Z$ are two smooth diffeomorphisms, then 
$$ \ell(\psi \circ \varphi) \leqslant  \ell(\varphi) + \ell(\psi). $$
\end{proposition}

\begin{proof}
We observe that we have injections of algebras of functions $$\psi^* \, : \, C^{\infty}(Z,\R_{\geq 0}) \hookrightarrow C^{\infty}(Y,\R_{\geq 0}) \mbox{ and }\psi^* \, : \, C^{\infty}(Z) \hookrightarrow C^{\infty}(Y).$$
It then suffices to use the identity $$ \frac{\zeta_{Z,a_0}}{\zeta_{X,\varphi^*\psi^*(a_0)}}  =  \frac{\zeta_{Z,a_0}}{\zeta_{Y,\psi^*(a_0)}}  \cdot  \frac{\zeta_{Y,\psi^*a_0}}{\zeta_{X,\varphi^*\psi^*(a_0)}}, $$ and similarly for the two-variable version. 
\end{proof} 

We cannot directly apply Lemma \ref{ddd} to conclude that $\ell$ induces a distance, but see Section \ref{fff}.

\section{Example: length of rescaling a circle} \label{circular}

\begin{se} Let $S_r$ denote the circle of radius $r$, which we parameterize by an angle $\theta \in [0,2\pi[$. The metric is $ds^2=r^2d\theta$, $g_{11}=r^2, g^{11}=r^{-2}$, the Laplacian is  $-r^{-2} \partial^2_\theta$, with spectrum $\{n^2r^{-2}\}_{n \in \Z_{>0}}$ with multiplicity two, and eigenspace for $n$ spanned by $\{\sin(n\theta),\cos(n\theta)\}$. Let $\zeta(s)$ denote the Riemann zeta function. One sees directly that $$\zeta_{S_r,a_0}= 2r^{2s+1} \left( \int_{0}^{2 \pi} a_0(\theta) d\theta \right) \zeta(2s)$$ and $$\zeta_{S_r,a_1,a_2}= 2 r^{2s-1} \left( \int_{0}^{2\pi}  a_1(\theta)\partial_\theta^2(a_2)(\theta) d\theta \right) \zeta(2s) .$$ 
Hence \bea \label{zetacirc} \zeta_{S_r,a_1,a_2}= r^{2}  \zeta_{S_r,a_1\partial_\theta^2(a_2)}.\eea 
\end{se}

\begin{se} Let us compute the length of the natural rescaling homeomorphism $$\varphi_{r_1,r_2} \, : \, S_{r_1} \rightarrow S_{r_2} \, \, : \, \, \theta \mapsto \theta \ \ \ (\theta \in [0,2\pi[).$$  We find 
$$ \left| \frac{\zeta_{S_{r_1},a^*_0}}{\zeta_{S_{r_2},a_0}}\right| =  (r_1/r_2)^{2s+1}\mbox{ and } \left| \frac{\zeta_{S_{r_1},a^*_1,a^*_2}}{\zeta_{S_{r_2},a_1,a_2}}\right|= (r_1/r_2)^{2s-1},$$
so we find for the length of $\varphi_{r_1,r_2}$: 
$$ \ell(\varphi_{r_1,r_2}) = \frac{1}{1+\frac{1}{5|\log(r_1/r_2)|}}.$$
Figure \ref{fig1} depicts the $\ell(\varphi_{r,1})$ for $0\leqslant  r \leqslant  2$. Observe the nice ``conformal'' symmetry $\ell(\varphi_{r_1,r_2})=\ell(\varphi_{r_2,r_1})$.  Also, the two-variable zeta function does not affect this computation (as is to be expected from the fact that the spectrum characterizes a circle); in the next section, we will consider isospectral tori, for which exactly the one-variable zeta function plays no role.

\begin{center}
\begin{figure}[h]
\includegraphics[width=10cm]{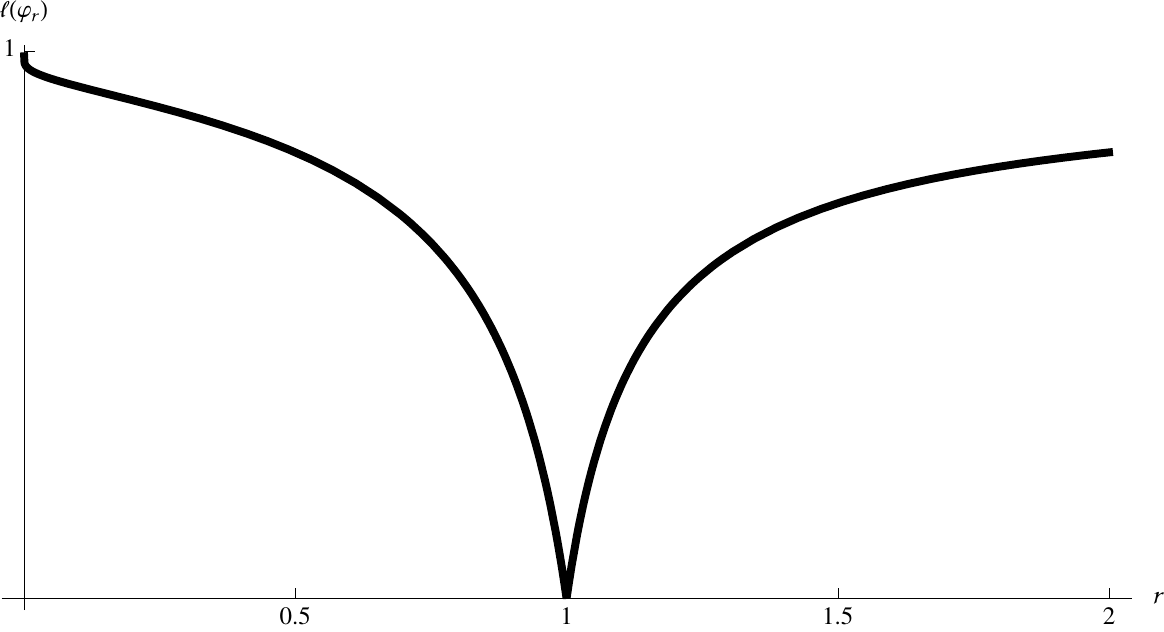}
\caption{Length of the rescaling homeomorphism $\varphi_r$ between a circle of radius $r$ and a circle of unit radius}
\label{fig1}
\end{figure}
\end{center}

\end{se}

\section{Example: length of a linear map between isospectral tori} \label{lengthtori}
\begin{se}Let $\T_1$ and $\T_2$ denote two \emph{isospectral} tori. Let $\varphi \, : \, \T_1 \rightarrow \T_2$ denote a smooth bijection, and assume that $\varphi$ arises from a linear map $A$ in the universal cover (any map of tori is homotopic to such a linear map with the same action on the homology of the torus, cf.\ \cite{Halpern}, Lemma 1): 
$$ \xymatrix{  \R^n \ar@{->}[r]^{A} \ar@{->}[d]_{\pi_1} &  \R^n   \ar@{->}[d]^{\pi_2} \\ \T_1=\R^n/\Lambda_1 \ar@{->}[r]^{\varphi} & \T_2=\R^n/\Lambda_2 } $$
This makes sense if $A\Lambda_1 \subseteq \Lambda_2$. If we denote by $G_1$ and $G_2$ the generator matrices of the two tori (matrices whose columns are basis vectors of the lattice), the condition is that \begin{equation} \label{intmat} G_2^{-1}AG_1 \in \mathrm{GL}(n,\Z). \end{equation} Taking determinants, we find $$w_\varphi = |\det(A)| = |\det(G_1^{-1} G_2)| = \vol(\T_2)/\vol(\T_1).$$
An example of such a map is the ``change of basis'' $A=G_2 G_1^{-1}$. Write $A^\top$ for the transpose of the matrix $A$. 
\end{se}

\begin{se} Since we assume $\T_1$ and $\T_2$ isospectral tori, they have the same (common) spectral zeta function. Hence from formula (\ref{torus1}) we find that 
$$ \left| \frac{\zeta_{\T_1,a^*_0}}{\zeta_{\T_2,a_0}}\right| =   \left| \frac{\int_{\T_2} a_0 w_{\varphi^{-1}}\, d\mu_{\R^n} }{\int_{\T_2} a_0 \, d\mu_{\R^n} } \right| = | \det(A^{-1}) | = \frac{\vol(\T_1)}{\vol(\T_2)} = 1.$$
Via formula (\ref{torus2}), the two variable zeta functions satisfy
$$ \sup_{\nabla a_1 \neq 0} \left| \frac{\tzeta_{\T_1,a^*_1}}{\tzeta_{\T_2,a_1}}\right| =  \sup_{\nabla a_1 \neq 0} \frac{\int_{\T_1} |\nabla(a^*_1)|^2 \, d\mu_{\R^n} }{\int_{\T_2} |\nabla(a_1)|^2\, d\mu_{\R^n} } =  \sup_{\nabla a_1 \neq 0} \frac{\int_{\T_1} |A\nabla(a_1)|^2\, d\mu_{\R^n} }{\int_{\T_2} |\nabla(a_1)|^2\, d\mu_{\R^n} }. $$
For every $v \in T_x \T_2$, we have $ |Av|^2 \leqslant ||A||_2 |v|^2 $, where $||A||_2$ is the spectral norm of the matrix $A$ ($=$ the square root of the largest eigenvalue of $A A^\top$).  Hence 
$$ d(\T_1,\T_2) \leqslant \ell(A) \leqslant \frac{\log ||A||_2}{1+\log ||A||_2}. $$
One may wonder whether this bound is attained. 
 \end{se}
\begin{example}
The smallest dimension in which there exist non-isometric isospectral tori is four, as was shown by Schiemann (\cite{Schiemann}), and an example is given by the two tori in the proof of \ref{exist}. For the specific map $A=G_- G_+^{-1}$  between these tori, we have $||A||_2 \approx 3.21537$, and $$d(\T_+, \T_-) \leqslant \ell(A) \leqslant 0.538733.$$
\end{example}

\section{Convergence in the spectral metric} \label{fff}

\begin{theorem} \label{convmet}
Suppose we are given two Riemannian manifolds $(X,g_X)$ and $(Y,g_Y)$ and a collection of smooth diffeomorphisms $\varphi_i \, : \, X \rightarrow Y$ whose length converges to zero. Then $X$ and $Y$ are isometric. 
\end{theorem}

\begin{proof} The proof is basically the ``convergent'' version of the first proof of Theorem \ref{Z0thm}. 

In the definition of $\ell(\varphi)$, we observe that both zeta functions $\zeta_{Y,a_0}(s)$ and $\zeta_{X,\varphi_i^*(a_0)}(s)$ have their right most pole at $s=d/2$. Both poles are simple, hence they cancel in the quotient. Therefore, the quotient function $\zeta_{X,\varphi_i^*(a_0)}(s)/\zeta_{Y,a_0}(s)$ is holomorphic in $s \geqslant d/2-1/2$. Also, since $a_0$ is positive, the quotient is a positive real valued function. We conclude from $\ell(\varphi) \rightarrow 0$ that 
$$\zeta_{X,\varphi_i^*(a_0)}(s)/\zeta_{Y,a_0}(s) \rightarrow 1 \mbox{ for } s \in \R.$$
In particular, we have convergence at $s=d/2$, and hence a convergence of residues (uniformly in $a_0$)
$$ \Res_{s=\frac d2} \zeta_{X,\varphi_i^*(a_0)}(s) = \lim_{s \rightarrow \frac d2+} \zeta_{X,\varphi_i^*(a_0)}(s) (s-\frac d2)  \rightarrow 
\lim_{s \rightarrow \frac d2+} \zeta_{Y,a_0}(s) (s-\frac d2) =  \Res_{s=\frac d2} \zeta_{Y,a_0}(s), $$
(not just in absolute value), where the limits are taken along the positive real axis.  

By the computation of these residues in Lemma \ref{residue}, we conclude that the jacobians converge to $1$: 
$$ w_{\varphi_i} \rightarrow 1. $$
For the two-variable zeta functions, one may reason in a similar way, using that $g(da,da)$ is a totally positive function. From Lemma \ref{res2}, we get in a similar way a convergence of metrics $$ \varphi_i^*(g_Y) \rightarrow g_X, $$ 
uniformly on $X$. 

Recall that the \emph{distortion} of a map $\varphi \, : \, X \rightarrow Y$ is defined to be
$$ \mathrm{dis}(\varphi) := \sup_{x_1,x_2 \in X} \left| d_Y(\varphi(x_1),\varphi(x_2))-d_X(x_1,x_2)\right|. $$
The distance in terms of the metric tensor is 
$$  d(x_1,x_2) := \mathop{\inf_{\gamma \in C^1([0,1],X)}}_{\gamma(0)=x_1, \gamma(1)=x_2} \int_0^1 \sqrt{ \sum g^{ij}(\gamma(t)) \gamma(t)'_i \gamma(t)'_j} \, dt   .$$ By uniform convergence of metric tensors on the manifold $X$, we can interchange the infimum in the definition of the distance with the limit in metrics to conclude that \begin{equation} \label{dis} \mathrm{dis}(\varphi_i) \rightarrow 0. \end{equation}
We can now finish the proof as in \cite{Burago} (proof of Thm.\ 7.3.30). Since $X$ is compact, we can find a dense countable set $S \subset X$, and we can find a subsequence $\{ \varphi'_i\}$ of $\{ \varphi_i \}$ that converges pointwise in $Y$ at every $x \in S$. This allows us to define a limit map $$\varphi \, : \, S \rightarrow Y \mbox{\ \ by \ \ } \varphi(x):= \lim \varphi'_i(x)$$  for $x \in S$. This limit map is distance-preserving by (\ref{dis}), and so can be extended to a distance-preserving bijection from $X \rightarrow Y$. Now the Myers-Steenrod theorem (\cite{KN}, 3.10) implies that $\varphi$ is a smooth isometry between $X$ and $Y$. 
\end{proof}

\begin{corollary}
The function ``zeta-distance'' $$d_\zeta(X,Y):=\min \{ \inf_{C^{\infty}(X \stackrel{\varphi}{\rightarrow} Y)} \ell(\varphi),+\infty\}$$
defines an extended metric between isometry classes of Riemannian manifolds. 
\end{corollary}

\begin{proof}
It suffices to prove that if $d_\zeta(X,Y)=0$, then $X$ and $Y$ are isometric, and this follows from the previous theorem. 
\end{proof}

\begin{remark} There are other distance functions between Riemannian manifolds up to isometry, such as Lipschitz, uniform or Gromov-Hausdorff distance (e.g., \cite{Gromov}, \cite{Burago}), the distances $d_t$ and $\delta_t$ of B\'erard-Besson-Gallot (\cite{BBG}) and the  spectral distances of Fukaya (\cite{Fukaya}) and Kasue-Kumura (\cite{KK1}).  These distances pose various computational challenges --- in the previous sections, we hope to at least have hinted at the computational aspects of the ``zeta-distance'' we propose. We finally observe that such distances play an increasingly important role in physics and cosmology (compare, e.g., \cite{MS0}, \cite{Douglas}).  
\end{remark} 

\begin{se} We conclude by comparing our ``zeta-distance'' $d_\zeta$ to the other distances. For this, we recall in the following diagram the relation between various forms of convergence:

\begin{center}
\begin{figure}[h]
$ \xymatrix{  & \mbox{Lipschitz-conv.} \ar@{=>}[dl]_{\mbox{\hspace*{-2cm} {\footnotesize simple spectrum}}}^{\mbox{\cite{BBG}}} \ar@{=>}[d]^{\mbox{\cite{BBG}}} \ar@{=>}[dr]^{\mbox{\cite{Burago}}}\\   d_t\mbox{-conv.} & \delta_t\mbox{-conv.} & \mbox{unif.\ conv.} \ar@{=>}[d]^{\mbox{\cite{Burago}}} \ar@{<=>}[r] & \mathbf{d_\zeta}\mbox{\textbf{-conv.}} \\  \mbox{Kasue-Kumura-conv.} \ar@{=>}[rr]_{\mbox{\cite{KK1}}} & &  \mbox{GH-conv.} &  } $
\caption{Some relations between convergence in various distances (in a fixed $C^{\infty}$-type)}
\label{figgraph}
\end{figure}
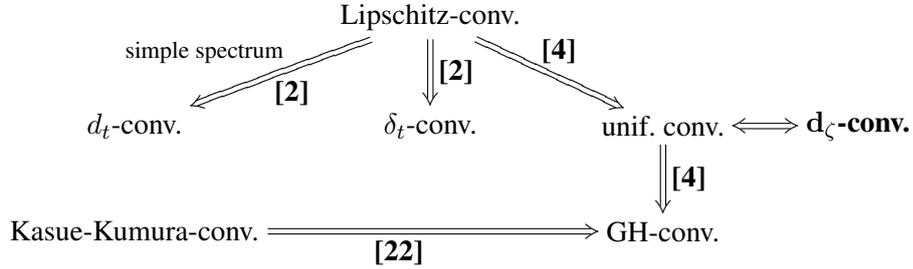
\end{center}

\end{se}

\begin{proposition}
Let $\mathcal{M}$ denote the set of closed Riemannian manifolds up to isometry. Then $d_\zeta$ induces the topology of uniform convergence in $C^{\infty}$-diffeomorphic types on $\mathcal{M}$, i.e., if two such manifolds are not $C^{\infty}$-diffeomorphic, then the manifolds are at infinite distance, and otherwise, a sequence of manifolds converge if and only if there is a sequence of $C^{\infty}$-diffeomorphisms between them whose distortion tends to zero. 
\end{proposition}

\begin{proof}
Suppose $(X_i,g_i) \rightarrow (X,g)$ converges in $d_\zeta$. This means that there is a sequence of $C^{\infty}$-diffeomorphisms $\varphi_i \, : \, (X_i,g_i) \rightarrow (X,g)$ whose length converges to zero. We precompose this with $\varphi_i^{-1}$: 
$$ \xymatrix{  (X_i,g_i) \ar@{->}[r]^{\varphi_i} \ar@{<-}[d]^{\varphi_i^{-1}} & (X,g) \\ (X,(\varphi_i^{-1})^*(g_i))  \ar@{->}[ru]_{\mbox{Id}} } $$
Hence we have a sequence of metrics $h_i:=(\varphi_i^{-1})^*(g_i)$ for which the length of the identity map converges to $0$. Taking residues in the two-variable zeta functions, we find that $h_i \rightarrow g$ uniformly. 

\end{proof}

\bibliographystyle{amsplain}

\begin{thebibliography}{10}

\bibitem{AKS}
Nachman Aronszajn, Andr{\'e} Krzywicki, and Jacek Szarski, \emph{A unique
  continuation theorem for exterior differential forms on {R}iemannian
  manifolds}, Ark. Mat. \textbf{4} (1962), 417--453 (1962).

\bibitem{BBG}
Pierre B{\'e}rard, G{\'e}rard Besson, and Sylvain Gallot, \emph{Embedding
  {R}iemannian manifolds by their heat kernel}, Geom. Funct. Anal. \textbf{4}
  (1994), no.~4, 373--398.

\bibitem{Berard}
Pierre~H. B{\'e}rard, \emph{Spectral geometry: direct and inverse problems},
  Lecture Notes in Mathematics, vol. 1207, Springer-Verlag, Berlin, 1986.

\bibitem{Burago}
Dmitri Burago, Yuri Burago, and Sergei Ivanov, \emph{A course in metric
  geometry}, Graduate Studies in Mathematics, vol.~33, American Mathematical
  Society, Providence, RI, 2001.

\bibitem{CF}
Luis~A. Caffarelli and Avner Friedman, \emph{Partial regularity of the zero-set
  of solutions of linear and superlinear elliptic equations}, J. Differential
  Equations \textbf{60} (1985), no.~3, 420--433.

\bibitem{ConnesLMP}
Alain Connes, \emph{Geometry from the spectral point of view}, Lett. Math.
  Phys. \textbf{34} (1995), no.~3, 203--238.

\bibitem{ConnesInv}
Alain Connes, \emph{A unitary invariant in {R}iemannian geometry}, Int. J. Geom.
  Methods in Modern Phys. \textbf{5} (2008), no.~8, 1215--1242.

\bibitem{CS}
John~H. Conway and Neil J.~A. Sloane, \emph{Four-dimensional lattices with the
  same theta series}, Internat. Math. Res. Notices (1992), no.~4, 93--96.

\bibitem{CM}
Gunther Cornelissen and Matilde Marcolli, \emph{Zeta functions that hear the
  shape of a {R}iemann surface}, J. Geom. Phys. \textbf{58} (2008), no.~5,
  619--632.

\bibitem{CM2}
Gunther Cornelissen and Matilde Marcolli, \emph{Quantum statistical mechanics,
  {$L$}-series and anabelian geometry}, arXiv:1009.0736

\bibitem{dJ}
Jan~Willem de~Jong, \emph{Graphs, spectral triples and {D}irac zeta functions},
  P-Adic Numbers, Ultrametric Analysis, and Applications \textbf{1} (2010),
  no.~4, 286--296.

\bibitem{Don}
Harold Donnelly, \emph{Bounds for eigenfunctions of the {L}aplacian on compact
  {R}iemannian manifolds}, J. Funct. Anal. \textbf{187} (2001), no.~1,
  247--261.

\bibitem{Douglas}
Michael~R. Douglas, \emph{Spaces of quantum field theories}, arxiv:1005.2779,
  to appear in {J.} {P}hys.: {C}onference {S}eries, {P}roceedings of {QTS6},
  2010.

\bibitem{Fukaya}
Kenji Fukaya, \emph{Collapsing of {R}iemannian manifolds and eigenvalues of
  {L}aplace operator}, Invent. Math. \textbf{87} (1987), no.~3, 517--547.

\bibitem{Gilkey}
Peter~B. Gilkey, \emph{Asymptotic formul{\ae} in spectral geometry}, Studies in
  Advanced Mathematics, Chapman \& Hall/CRC, Boca Raton, FL, 2004.

\bibitem{Gromov}
Mikhael Gromov, \emph{Metric structures for {R}iemannian and non-{R}iemannian
  spaces}, Modern Birkh\"auser Classics, Birkh\"auser Boston Inc., Boston, MA,
  2007.

\bibitem{Halpern}
Benjamin Halpern, \emph{Periodic points on tori}, Pacific J. Math. \textbf{83}
  (1979), no.~1, 117--133.

\bibitem{HL}
Qing Han and Fang-Hua Lin, \emph{Nodal sets of solutions of elliptic
  differential equations}, Monograph in preparation, available at
  http://www.nd.edu/\~{}qhan/, 2009.

\bibitem{HardyDirichlet}
Godfrey~H. Hardy and Marcel Riesz, \emph{The general theory of {D}irichlet's
  series}, Cambridge Tracts in Mathematics and Mathematical Physics, No. 18,
  Stechert-Hafner, Inc., New York, 1964.

\bibitem{Higsonzeta}
Nigel Higson, \emph{Meromorphic continuation of zeta functions associated to
  elliptic operators}, Operator algebras, quantization, and noncommutative
  geometry, Contemp. Math., vol. 365, Amer. Math. Soc., Providence, RI, 2004,
  pp.~129--142.

\bibitem{Ikeda}
Akira Ikeda, \emph{On lens spaces which are isospectral but not isometric},
  Ann. Sci. \'Ecole Norm. Sup. (4) \textbf{13} (1980), no.~3, 303--315.

\bibitem{KK1}
Atsushi Kasue and Hironori Kumura, \emph{Spectral convergence of {R}iemannian
  manifolds}, Tohoku Math. J. (2) \textbf{46} (1994), no.~2, 147--179.

\bibitem{KN}
Shoshichi Kobayashi and Katsumi Nomizu, \emph{Foundations of differential
  geometry. {V}ol. {I}}, Wiley Classics Library, John Wiley \& Sons Inc., New
  York, 1996.

\bibitem{LR2}
Giovanni Landi and Carlo Rovelli, \emph{General relativity in terms of {D}irac
  eigenvalues}, Phys. Rev. Lett. \textbf{78} (1997), no.~16, 3051--3054.

\bibitem{LR1}
Giovanni Landi and Carlo Rovelli, \emph{Gravity from {D}irac eigenvalues}, Modern Phys. Lett. A
  \textbf{13} (1998), no.~6, 479--494.

\bibitem{Milnor}
John Milnor, \emph{Eigenvalues of the {L}aplace operator on certain manifolds},
  Proc. Nat. Acad. Sci. U..S.A. \textbf{51} (1964), 542.1.

\bibitem{Schiemann}
Alexander Schiemann, \emph{Ternary positive definite quadratic forms are
  determined by their theta series}, Math. Ann. \textbf{308} (1997), no.~3,
  507--517.

\bibitem{MS0}
Masafumi Seriu, \emph{Spectral representation of the spacetime structure: the
  ``distance'' between universes with different topologies}, Phys. Rev. D (3)
  \textbf{53} (1996), no.~12, 6902--6920.

\bibitem{Serre}
Jean-Pierre Serre, \emph{Cours d'arithm{\'e}tique}, 4i{\`e}me ed., Presses
  Universitaires de France, Paris, 1995.

\bibitem{Uh1}
Karen Uhlenbeck, \emph{Eigenfunctions of {L}aplace operators}, Bull. Amer.
  Math. Soc. \textbf{78} (1972), 1073--1076.

\bibitem{Vigneras}
Marie-France Vign{\'e}ras, \emph{Vari\'et\'es riemanniennes isospectrales et
  non isom\'etriques}, Ann. of Math. (2) \textbf{112} (1980), no.~1, 21--32.

\bibitem{Watson}
Bill Watson, \emph{Manifold maps commuting with the {L}aplacian}, J.
  Differential Geometry \textbf{8} (1973), 85--94.

\bibitem{Zelditchmult}
Steven Zelditch, \emph{On the generic spectrum of a {R}iemannian cover}, Ann.
  Inst. Fourier (Grenoble) \textbf{40} (1990), no.~2, 407--442.

\bibitem{ZelditchCFO}
Steven Zelditch, \emph{Isospectrality in the {FIO} category}, J. Differential Geom.
  \textbf{35} (1992), no.~3, 689--710.

\end{thebibliography}

\end{document}